\documentclass{article}

\usepackage{pgfgantt} 

\usepackage{xcolor}
\definecolor{parchment}{RGB}{214, 204, 169}

\usepackage{microtype}
\usepackage{verbatim}
\usepackage{graphicx}
\usepackage{rotating}

\usepackage{svg}

\usepackage{url}
\usepackage{hyperref}

\usepackage{pslatex}

\usepackage{datetime}
\newdateformat{versiondate}{%
\THEMONTH\THEDAY}

\usepackage{pgf,tikz,environ}
\usetikzlibrary{arrows}
\usetikzlibrary{cd}

\usepackage{xpatch} 

\usepackage{amsthm}
\usepackage[fleqn]{amsmath} 
\usepackage{amsfonts}
\usepackage{amssymb}

\newtheoremstyle{zoltanstyle}
  {1em} 
  {\topsep} 
  {} 
  {} 
  {\bfseries} 
  {.} 
  {.5em} 
  {} 

\theoremstyle{zoltanstyle}
\swapnumbers
\xpatchcmd\swappedhead{~}{.~}{}{}
\newtheorem{body}{}
\numberwithin{body}{section}

\newtheorem{corollary}[body]{Corollary}
\newtheorem{definition}[body]{Definition}
\newtheorem{example}[body]{Example}

\newtheorem{lemma}[body]{Lemma}
\newtheorem{problem}[body]{Problem}
\newtheorem{proposition}[body]{Proposition}

\newtheorem{theorem}[body]{Theorem}

\expandafter\let\expandafter\oldproof\csname\string\proof\endcsname
\let\oldendproof\endproof

\renewenvironment{proof}[1][\proofname]{%
  \oldproof[\normalfont \bfseries #1.]%
}{\oldendproof}

\usepackage{enumitem}
\usepackage{mathptmx}

\usepackage[T1]{fontenc}
\usepackage[utf8]{inputenc}

\usetikzlibrary {positioning}
\usetikzlibrary{fadings}
\usetikzlibrary{decorations.pathreplacing}
\usetikzlibrary{decorations.pathmorphing}
\tikzset{snake it/.style={decorate, decoration=snake}}
\definecolor{gold}{RGB}{255,215,0}
\definecolor{softBlack}{RGB}{45, 47, 49}
\definecolor{creamWhite}{RGB}{245,244,241}
\definecolor{softGray}{RGB}{220, 216, 214}
\definecolor{brick}{RGB}{232, 48, 48}

\newcommand{\typesetoperator}[1]{\mathbf{#1}}

\DeclareMathOperator{\ds}{\typesetoperator{ds}}
\DeclareMathOperator{\exclmid}{\typesetoperator{EM}}
\DeclareMathOperator{\doubleneg}{\typesetoperator{DN}}
\newcommand{\SetComp}[2]{\left\{ {#1}\:\middle|\:{#2} \right\}}   
\newcommand{\pita}{\mathrel{\reflectbox{\rotatebox[origin=c]{180}{$\mathbb{A}$}}\!\!}}  

\usepackage{xspace}
\definecolor{gold}{RGB}{255,215,0}
\definecolor{softBlack}{RGB}{45, 47, 49}
\definecolor{creamWhite}{RGB}{245,244,241}
\definecolor{softGray}{RGB}{220, 216, 214}
\definecolor{brick}{RGB}{232, 48, 48}

\usepackage{enumitem}
\title{Degree of Satisfiability in Heyting Algebras}
\author{Benjamin Merlin Bumpus \thanks{University of Florida, Gainesville FL32611, United States of America.}   ~ \&  Zoltan A. Kocsis \thanks{University of New South Wales, Kensington NSW 2052, Australia.}}
\date{\today}

\usepackage[normalem]{ulem} 

\begin{document}

\maketitle

\begin{abstract}
Given some finite structure $M$ and property $p$, it is natural to study the degree of satisfiability of $p$ in $M$; i.e. to ask: what is the probability that uniformly randomly chosen elements in $M$ satisfy $p$? In group theory, a well-known result of Gustafson states that the equation $xy=yx$ has a finite satisfiability gap: its degree of satisfiability is either $1$ (in Abelian groups) or no larger than $\frac{5}{8}$. Degree of satisfiability has proven useful in the study of (finite and infinite) group-like and ring-like algebraic structures, but finite satisfiability gap questions have not been considered in lattice-like, order-theoretic settings yet.

Here we investigate degree of satisfiability questions in the context of Heyting algebras and intuitionistic logic. We classify all equations in one free variable with respect to finite satisfiability gap, and determine which common principles of classical logic in multiple free variables have finite satisfiability gap. In particular we prove that, in a finite non-Boolean Heyting algebra, the probability that a randomly chosen element satisfies $x \vee \neg x = \top$ is no larger than $\frac{2}{3}$. Finally, we generalize our results to infinite Heyting algebras, and present their applications to point-set topology, black-box algebras, and the philosophy of logic.
\end{abstract}

\section{Introduction}

\begin{definition}\label{def:ds-itself}
Take a first-order language $\mathcal{L}$, a finite $\mathcal{L}$-structure $M$, and an $\mathcal{L}$-formula $\varphi(x_1,\dots,x_n)$ in $n$ free variables. We call the quantity
$$ \frac{ \left   \lvert    \SetComp{(a_1, \dots, a_n) \in M^n}{\varphi(a_1, \dots,a_n)} \right   \lvert    }{    \lvert    M      \lvert     ^n }$$
the \textit{degree of satisfiability} of the formula $\varphi$ in the structure $M$, and denote it $\ds_M(\varphi)$.
\end{definition}

\begin{minipage}{\linewidth}
\begin{definition}\label{def:finite-gap}
Take a theory $T$ over a first-order language $\mathcal{L}$, along with an $\mathcal{L}$-formula $\varphi$. If we can find a constant $\varepsilon > 0$ such that for every finite model $M$ of the theory $T$, we have either

\begin{enumerate}
    \item $\ds_M(\varphi) = 1$; or else 
    \item $\ds_M(\varphi) \leq 1 - \varepsilon,$
\end{enumerate}
then we say that the formula $\varphi$ has \textit{finite satisfiability gap} $\varepsilon$ in $T$.
\end{definition}

\end{minipage}

\begin{body}\label{body:group-discussion}
A classic result of Gustafson~\cite{gustafson-degree-of-commutativity} states that in a finite non-Abelian group $G$, $\ds_G(xy=yx)$ cannot exceed $\frac{5}{8}$. In particular, ``deceptive'' groups, which only barely fail to be Abelian, do not exist. The seminal work of Antol\'{i}n, Martino and Ventura~\cite{antolin-infinite-dc}, generalizing Gustafson's result to a class of finitely generated groups, boosted collective interest in finite gap properties of other group-theoretic equations and formulae: in the last few years, tight finite gap bounds have been obtained for the nested simple commutator equation~\cite{lescot-nilpotency}, the 2-Engel and metabelian identities~\cite{delizia-satisfiability}, the equations $xy^2=y^2x$ and $xy^3=y^3x$~\cite{zoltanGroups}, and others. The existence of an equation in the language of group theory that does not have finite satisfiability gap remains open, even in the case of equations in only one free variable. An elementary argument shows that $x^2 = 1$ has finite satisfiability gap $\frac{1}{4}$, and Laffey~\cite{laffey-powers-of-three} established a finite gap of $\frac{2}{9}$ for $x^3 = 1$. For $x^p = 1$ with $p \geq 5$, only partial results are known.
\end{body}

\subsection{Degrees of Classicality}

\begin{body}
Recall that a Heyting algebra $\langle H, \wedge, \vee, \rightarrow, \bot, \top \rangle$ is a bounded distributive lattice where for every two elements $a, b \in H$, the set $\SetComp{c \in H}{a \wedge c \leq b}$ has a distinguished maximum element denoted $a \rightarrow b$. We introduce $\neg x$ as shorthand for $x \to \bot$. A Boolean algebra is a Heyting algebra $\langle H, \wedge, \vee, \bot, \top \rangle$ in which $x \to y = \neg x \lor y$ for all $x, y \in H$.
We will assume that $\bot \neq \top$ throughout this article, since the trivial Heyting algebra satisfies every equation.
\end{body}

\begin{body}
The main significance of Heyting algebras is that they provide the natural algebraic semantics for (the propositional fragment of) intuitionistic logic. We leverage this correspondence multiple times throughout this paper, by freely identifying propositional formulae of intuitionistic logic with first-order \textit{terms} in the language of Heyting algebras. Recall that a propositional formula $\varphi$ is provable in intuitionistic logic precisely if $H \models (\varphi = \top)$ for every Heyting algebra~$H$.
\end{body}

\begin{definition}
Consider the language $\langle \wedge, \vee, \rightarrow, \bot, \top \rangle$ of Heyting algebras. We call an equation $\varphi$ a \textit{classical principle} if a Heyting algebra $H$ satisfies $H \models \varphi$ precisely when $H$ is a Boolean algebra.
\end{definition}

\begin{body}
Notice that an equation $f = \top$ constitutes a classical principle precisely if adding $f$ as an axiom to the standard Hilbert calculus for intuitionistic propositional logic yields a Hilbert calculus whose tautologies are those of classical propositional logic.
\end{body}

\pagebreak

\begin{example}\label{example:classical-principles}
The following are well-known classical principles:
\begin{enumerate}
    \item $x \vee \neg x = \top$  (\textit{the law of excluded middle}),
    \item $\neg\neg x = x$ (\textit{double-negation elimination}),
    \item $(x \rightarrow y) \rightarrow x = x$ (\textit{Peirce's law}),
    \item $\neg y \rightarrow \neg x = x \rightarrow y$ (\textit{contrapositive principle}),
    \item $(\neg x \rightarrow y) \rightarrow (x \rightarrow y) \rightarrow y = \top$ (\textit{LEM - eliminator form}),
    \item $x \rightarrow y = \neg x \vee y$ (\textit{material implication}).
\end{enumerate}
\end{example}

\begin{body}
In what follows, we investigate the degree of satisfiability of formulae in Heyting algebras. We give a complete classification of all equations in one free variable with respect to finite satisfiability gap in Section~\ref{sec:1-var} (as explained in \ref{body:group-discussion}, an analogous classification for equations in groups remains elusive). In particular, we obtain that the law of excluded middle, $x \vee \neg x = \top$, either holds for all $x$, or for no more than $\frac{2}{3}$ of all $x$. While the poorly understood structure of free Heyting algebras makes a similar classification for two-variable equations unlikely, we provide a thorough treatment of the classical principles enumerated above, along with some results that can be used to establish (the lack of) finite gap in many two-variable cases (Section~\ref{sec:2-var}). Finally, we generalize the classification of one variable equations to infinite Heyting algebras (Section \ref{sec:infinite}), and present applications of our results in point-set topology (Corollary \ref{cor:topology}), in black-box algebras (Section \ref{sec:black-box}), and in the philosophy of logic (Section \ref{sec:discussion}).
\end{body}

\subsection{Technical Preliminaries}

\begin{definition}
The \emph{upset} of an element $x$ in a Heyting algebra $H$ (or any poset) is the set $\SetComp{y \in H}{y \geq x}$, which we denote by $\uparrow_H x$ (or simply $\uparrow x$ if $H$ is clear from context).
\end{definition}

\begin{definition}
Given a Heyting algebra $H$ with bottom element $\bot$ and top element $1$, the Heyting algebra $H \oplus \top$ is the Heyting algebra obtained from $H$ by adding a new element $\top$ to $H$ and the relation $\top \geq x$ for all $x \in H$. The operation of \emph{adjoining a $k$-chain} to $H$ is denoted $H \oplus_{k} \top$ and is defined recursively as $H \oplus_{k} \top := (H \oplus_{k -1} \top) \oplus \top$. 
\end{definition}

\begin{proposition}\label{prop:eqn-base-form}
Every finite system of equations in the language of Heyting algebras is equivalent to some equation of the form $\varphi = \top$.
\begin{proof}
Let $a,b$ denote arbitrary terms in the language of Heyting algebras. Use the fact that $a = b$ precisely if $a \rightarrow b = \top$ and $b \rightarrow a = \top$, along with the fact that $a = \top$ and $b = \top$ hold precisely if $a \wedge b = \top$ holds, to reduce the system to a single equation in the given form.
\end{proof}
\end{proposition}

\begin{body}
In accordance with Proposition~\ref{prop:eqn-base-form}, a complete classification of all equations in a given number of variables with respect to degree of satisfiability immediately induces a corresponding classification for all finite systems of equations as well.
\end{body}

\begin{proposition}\label{prop:ds-of-product-is-product-of-ds}
Consider an algebraic theory $T$ over a language $\mathcal{L}$. For any two models $H,J$ of $T$, $H \times J$ is also a model of $T$. Moreover, when $H$ and $J$ are finite, the equality $$\ds_{H \times J}{\varphi} =  \ds_{H}{\varphi} \times \ds_{J}{\varphi}$$ holds for any $\mathcal{L}$-equation $\varphi$.
\end{proposition}

\begin{body}
Proposition~\ref{prop:ds-of-product-is-product-of-ds} does not generalize to arbitrary formulae in the language of Heyting algebras. Consider the join-irreducibility formula $\varphi$ defined as an abbreviation for $\forall y, z. (x = y \vee z) \rightarrow (x = y) \vee (x = z)$. In a Boolean algebra, only $\bot$ and atoms are join-irreducible, so $\varphi$ has degree of satisfiability $1$ in the 2-element Boolean algebra, degree of satisfiability $\frac{3}{4}$ in the 4-element Boolean algebra, but degree $\frac{1}{2} \neq 1 \cdot \frac{3}{4}$ in the 8-element Boolean algebra.
\end{body}

\subsection{Excluded-middle and double-negation-elimination}
\begin{definition} \label{def:basic-loci}
Take a Heyting algebra $H$. We let $\exclmid_H$ denote the center of $H$, i.e. the set $\SetComp{x \in H}{x \vee \neg x = \top}$ of elements that satisfy the law of excluded middle. Similarly, we let $\doubleneg_H$ denote the set $\SetComp{x \in H}{\neg\neg x = x}$ of elements that satisfy double-negation elimination.
\end{definition}

\begin{body}
Since the sets $\exclmid_H$ and $\doubleneg_H$ are of particular importance to us, we take a few moments to establish and/or recall some of their simple properties.
\end{body}

\begin{proposition}\label{prop:excluded_mid_largest_BA}
For any Heyting algebra $H$, the set $\exclmid_H$ is a sub-algebra of $H$; in fact it is a Boolean algebra and, for every other Boolean sub-algebra $B$ of $H$, we have $B \subseteq \exclmid_H$.
\end{proposition}

\begin{theorem}[Glivenko \cite{glivenko-main}]\label{thm:glivenko}
The subposet $\doubleneg_H$ of a Heyting algebra $H$ always forms a Boolean algebra; furthermore, it is a $\wedge$-subsemilattice of $H$.
\end{theorem}

\begin{body}\label{body:doubleneg-not-subalgebra}
The containment $\exclmid_H \subseteq \doubleneg_H$ holds in any Heyting algebra $H$. Moreover, if $\exclmid_H \neq \doubleneg_H$, then Proposition~\ref{prop:excluded_mid_largest_BA} along with Theorem~\ref{thm:glivenko} gives that $\doubleneg_H$ cannot be a sub-algebra of $H$.
\end{body}

\begin{proposition}\label{prop:if_big_excluded_middle_then_doubleneg_is_subalgebra}
Let $H$ be a finite Heyting algebra. If $\ds_H(x \lor \neg x = \top) \geq 1/2$, then $\exclmid_H = \doubleneg_H$.
\end{proposition}
\begin{proof}
Since $\exclmid_H$ and $\doubleneg_H$ are both Boolean algebras, we can find numbers $n,m \in \mathbb{N}$ such that $   \lvert   \exclmid_H   \lvert    = 2^n$ and $   \lvert   \doubleneg_H   \lvert    = 2^m$. Since $\exclmid_H \subseteq \doubleneg_H$, we must have either $n = m$ (as desired) or \[   \lvert   \doubleneg_H   \lvert    \geq 2   \lvert   \exclmid_H   \lvert    = 2 \ds_H(x \lor \neg x = \top)    \lvert   H   \lvert    \geq    \lvert   H   \lvert   .\]
This means that $\doubleneg_H = H$ and so double-negation elimination holds everywhere. But then $H$ is a Boolean algebra which implies that the law of excluded middle holds everywhere as well; in other words $\doubleneg_H = H = \exclmid_H$, as desired.
\end{proof}

\begin{body}
We note that the $\frac{1}{2}$ bound obtained in Proposition~\ref{prop:if_big_excluded_middle_then_doubleneg_is_subalgebra} is not tight. A tight bound~($\frac{2}{5}$) follows immediately from Theorem~\ref{thm:strong-relativism} in Section~\ref{sec:discussion}, using the fact that the smallest Heyting algebra $H$ with $\doubleneg_H \not\subseteq \exclmid_H$ has 5 elements.
\end{body}

\section{Equations in one free variable}\label{sec:1-var}

\begin{body}
In this section we classify all equations in one free variable into two classes: those equations which have finite satisfiability gap and those which do not. The main result of this section, Theorem~\ref{thm:1-var-classification}, states that (up to logical equivalence of first-order formulae) only three equations, $p = \top, p = \bot$, and $p \vee \neg p = \top$ belong to the first class.
\end{body}

\begin{proposition}\label{prop:1-var-trivial}
The equations $p = \top$ and $\neg p = \top$ have finite satisfiability gap $1/2$.
\begin{proof}
Since every Heyting algebra $H$ has at least two elements, we have that $$\ds_{H}(p = \top) = \frac{ \left   \lvert   \SetComp{y \in H}{y = \top}\right   \lvert     }{   \lvert   H   \lvert   } = \frac{1}{   \lvert   H   \lvert   } \leq \frac{1}{2}.$$ Similarly, noticing that $\neg p = \top$ holds only if $p = \bot$, we get that $\ds_{H}(\neg p = \top) \leq \frac{1}{2}$ as well. Both of these gaps are realized in the Heyting algebra with two elements.
\end{proof}
\end{proposition}

\begin{body}
To establish an analogous result for the law of excluded middle (Theorem~\ref{thm:excluded-middle-gap}), we will argue inductively on the multiplicative structure of Heyting algebras by making use of Proposition~\ref{prop:ds-of-product-is-product-of-ds} (recall that this relates the degree of satisfiability of an equation in a product of Heyting algebras to its degree of satisfiability in the factors).
\end{body}

\begin{theorem}\label{thm:excluded-middle-gap}
The equation $p \vee \neg p = \top$ has finite satisfiability gap $1/3$.
\begin{proof}
We proceed by simultaneous induction on the number of elements of the finite algebra $H$ and the number of elements of the set $\exclmid_H$. In the base case, the algebra $H$ satisfies $   \lvert   \exclmid_H   \lvert    = 2$. Consequently, $\ds_H(p \vee \neg p = \top) = \frac{2}{   \lvert   H   \lvert   }$, which is either $1$ or at most $\frac{2}{3}$, as desired.

In the inductive case, the algebra $H$ has $   \lvert   \exclmid_H   \lvert    > 2$.
Since $   \lvert   \exclmid_H   \lvert    > 2$, we can find some $c \in \exclmid_H$ such that $c \neq \bot$ and $c \neq \top$. Consider the subposet $H_c = \SetComp{x \in H}{x \leq c}$ of $H$. This clearly forms a Heyting algebra with the order inherited from $H$.

Define the map $f:H \to H_c \times H_{\neg c}$ by $f(x)=(x \wedge c, x \wedge \neg c)$. Then $f$ is clearly a bounded lattice homomorphism. It is injective because, if $f(x) = f(y)$, then 
\[
x = x \wedge \top
  = x \wedge (c \vee \neg c ) 
  = (x \wedge c) \vee (x \wedge \neg c) 
  = (y \wedge c) \vee (y \wedge \neg c) 
  = y.
\]
Moreover, $f$ is surjective because if $(a,b) \in H_c \times H_{\neg c}$, then 
\begin{align*}
    f(a \vee b) &= ((a \vee b) \wedge c, (a \vee b) \wedge \neg c)\\
    &=((a \wedge c) \vee (b \wedge c), (a \wedge \neg c) \vee (b \wedge  \neg c)) \\
    &=(a \vee (b \wedge c), (a \wedge \neg c) \vee b) \\
    &= (a \vee \bot, \bot \vee b) &(\text{since } c \text{ is in the center})\\
    &=(a,b)
\end{align*}

By the finiteness of $H$, it follows that $f$ preserves implications as well, and is therefore a Heyting algebra isomorphism.

Since $   \lvert   H_c   \lvert    <    \lvert   H   \lvert   $ and $   \lvert   H_{\neg c}   \lvert    <    \lvert   H   \lvert   $, the inductive hypothesis applies, and gives us that either $\ds_{H_c}(p \vee \neg p = \top) = 1$ and $\ds_{H_{\neg c}}(p \vee \neg p = \top) = 1$ both hold, or at least one of $\ds_{H_c}(p \vee \neg p = \top) \leq \frac{2}{3}$ or $\ds_{H_{\neg c}}(p \vee \neg p = \top) \leq \frac{2}{3}$ holds. In either case, applying Proposition~\ref{prop:ds-of-product-is-product-of-ds} concludes the proof.
\end{proof}
\end{theorem}

\begin{body}
The $\frac{1}{3}$ bound for the satisfiability gap of the equation $x \vee \neg x = \top$ is tight and indeed it is realized in the three-element Heyting algebra. Note that Theorem~\ref{thm:lem-finite-infinite} gives rise to a substantially different proof of Theorem~\ref{thm:excluded-middle-gap}. 
\end{body}

\begin{figure}
    \centering
\begin{tikzcd}
                                                    &                                                                          & \top \arrow[d, no head, dotted]                                                 &     &                                    &                              \\
                                                    &                                                                          & {}                                                                              &     &                                    &                              \\
                                                    & {}                                                                       &                                                                                 & {}  &                                    & i_\infty = \top = d_\infty   \\
i_3 \arrow[rd, no head] \arrow[ru, no head, dotted] &                                                                          & d_3 \arrow[rd, no head] \arrow[lu, no head, dotted] \arrow[ru, no head, dotted] &     & i_{n+1} = i_n \to d_n              & d_{n+1} = i_n \vee d_n       \\
                                                    & \mathbf{d_2} \arrow[rd, no head] \arrow[ld, no head] \arrow[ru, no head] &                                                                                 & i_2 & i_2 = \neg \neg p                  & \mathbf{d_2} = p \vee \neg p \\
\mathbf{i_1} \arrow[rd, no head]                    &                                                                          & \mathbf{d_1} \arrow[ru, no head]                                                &     & \mathbf{i_1} = \neg p              & \mathbf{d_1} = p             \\
                                                    & \mathbf{\bot} \arrow[ru, no head]                                        &                                                                                 &     & \mathbf{i_0} = \bot = \mathbf{d_0} &                             
\end{tikzcd}
\caption{(\textbf{LEFT}) Hasse diagram of the Rieger-Nishimura lattice.
(\textbf{RIGHT}) Recursive definition of the terms of the Rieger-Nishimura lattice. Theorem~\ref{thm:1-var-classification} states that the equations $i_0 = \top$, $i_1 = \top$, $d_1 = \top$, $d_2 = \top$ and $i_\infty = \top$ (corresponding elements marked in \textbf{bold}) are the only ones with finite satisfiability gap.
}
\label{fig:RN_Lattice}
\end{figure}

\begin{body}
The formulae in one free variable which are not logically equivalent to those that occur in Proposition~\ref{prop:1-var-trivial} and Theorem~\ref{thm:excluded-middle-gap} do not have finite satisfiability gap. To show this (Lemma~\ref{lemma:1-var-negative-direction}) we need to recall the definition of the Rieger-Nishimura lattice (Figure~\ref{fig:RN_Lattice}).
\end{body}

\begin{definition}\label{def:rieger-nishimura-lattice}
We define the sequences $d$ and $i$ of \emph{disjunctive} and \emph{implicative Rieger-Nishimura formulae} in the free variable $p$ by mutual recursion as follows:
\begin{align*}
    & d_0 = \bot, & \quad & i_0 = \bot, \\
    & d_1 = p, & \quad & i_1 = \neg p, \\
    & d_{n+1} = i_n \vee d_n, & \quad & i_{n+1} = i_n \rightarrow d_n. 
\end{align*}
The \emph{Rieger-Nishimura lattice} consists of the formulae $\top, d_n, i_n$ for all $n \in \mathbb{N}$, equipped with the ordering defined by $a \leq_{RN} b$ precisely if $a \rightarrow b$ holds in intuitionistic propositional logic.
\end{definition}

\begin{theorem}[Rieger \cite{rieger-nishimura-lattice}]\label{thm:rieger-nishimura}
As a Heyting algebra, the Rieger-Nishimura lattice is isomorphic to the free Heyting algebra on one generator.
\end{theorem}

\begin{body}
It follows from Theorem~\ref{thm:rieger-nishimura} that every system of equations in one free variable $p$ is logically equivalent to an equation of the form $\varphi(p) = \top$, where $\varphi$ belongs to the Rieger-Nishmiura lattice. Notice that, in the notation of Definition~\ref{def:rieger-nishimura-lattice}, Proposition~\ref{prop:1-var-trivial} and Theorem~\ref{thm:excluded-middle-gap} show that the equations $i_0 = \top$, $i_1 = \top$, $d_1 = \top$ and $d_2 = \top$ all have finite satisfiability gap.
\end{body}

\begin{proposition}\label{prop:no-gap-1-var-easy}
The equation $\neg \neg p = \top$ has no finite satisfiability gap. Furthermore, we have $\ds_{H \oplus_k \top}(\neg \neg p = \top) < 1$ and $\lim_{k \to \infty}\ds_{H \oplus_k \top}(\neg \neg p = \top) = 1$ for any Heyting algebra $H$ and $k \in \mathbb{N}$.
\end{proposition}
\begin{proof}
If $\neg \neg p = \top$, then $\bot = \neg p \wedge \neg \neg p = \neg p \wedge \top = \neg p$.
Since every element of the $n$-element chain, with the exception of $\bot$ itself, negates to bottom (for any $n$), the sequence $(C_n)_{n\in \mathbb{N}}$ of chains witnesses the fact that $\neg \neg p = \top$ has no finite satisfiability gap.  Furthermore, for any Heyting algebra $H$, every element $H \oplus_{k} \top \setminus H$ also negates to bottom (but $\bot$ itself does not); thus the rest of the claim follows by the same argument as above.
\end{proof}

\begin{body}
Next consider $i_3$ which is equivalent to $\neg \neg p \to p$ and the equation $i_3 = \top$ (i.e. the double-negation-elimination equation). Since both double-negation elimination and the law of excluded middle are classical principles, it is perhaps surprising to discover that, while $p \vee \neg p = \top$ has a gap (Theorem~\ref{thm:excluded-middle-gap}), the equation $\neg \neg p = p$ does not (Corollary~\ref{cor:dneg-has-no-gap}).
\end{body}

\begin{proposition}\label{prop:image-of-neg-is-dneg-set}
In every Heyting algebra $H$, the image of the negation map $x \mapsto \neg x$ coincides with the set $\doubleneg_H$.
\begin{proof}
The identity $\neg\neg\neg x = \neg x$ holds in every Heyting algebra, so every $y \in H$ in the image of the negation map satisfies $\neg\neg y = \neg\neg\neg x = \neg x = y$. Conversely, if an element $x \in H$ satisfies $\neg\neg x = x$, then $x$ is the image of $\neg x$ under the negation map.
\end{proof}
\end{proposition}

\begin{lemma}\label{lemma:neg-adjoin-top}
For any Heyting algebra $H$, let $\iota: H \hookrightarrow H \oplus \top$ be the obvious lattice inclusion of $H$ into $H \oplus \top$ (i.e. such that $(H \oplus \top ) \setminus \iota(H) = \{\top\}$). Then, for any $x \in H$ which is not $\bot_H$, we have $\neg  \iota( x) = \iota( \neg x)$.
\begin{proof}
If $\neg \iota(x) = \max \SetComp{c \in H \oplus T}{\iota(x) \wedge c = \bot} = \top$, then $\iota(x) = \iota(x) \wedge \top = \bot$ and hence $x = \bot_{H}$ since $\iota$ is an injection. Thus, for any $x \in H \setminus \{\bot\}$, there is a $y \in H \setminus \{\bot\}$ such that $\neg \iota(x) = \iota(y)$ and so we must have $y = \neg x$, as desired. 
\end{proof}
\end{lemma}

\begin{corollary}\label{cor:dneg-has-no-gap}
The formula $\neg\neg x = x$ has no finite satisfiability gap.
\end{corollary}
\begin{proof}
Consider the lattice inclusion $\iota: B_n \hookrightarrow B_n \oplus \top$ as in Lemma~\ref{lemma:neg-adjoin-top} of a Boolean algebra $B_n$. It is enough to show that $\iota(\top_{B_n})$ is the only element of $B_n \oplus T$ for which double negation elimination does not hold since then one has that \[\lim_{n \to \infty}\bigl( \ds_{B_n \oplus \top}(\neg \neg x = x) \bigr)= \lim_{n \to \infty}\Bigl( \frac{\lvert B_n \rvert}{\lvert B_n \oplus \top \rvert} \Bigr) = 1.\]  To that end, using Lemma~\ref{lemma:neg-adjoin-top}, we have that $\neg \neg \iota(\top_{B_n}) = \neg \bot = \top \neq \iota(\top_{B_n})$. Conversely $\neg \neg \iota(\bot_{B_n}) = \neg \neg \bot = \bot = \iota(\bot_{B_n})$ and, for all $b \in B_n \setminus \bot_{B_n}$, we have $\neg \neg \iota(b) = \iota(\neg \neg b) = \iota(b)$ by Lemma~\ref{lemma:neg-adjoin-top} and since $B_n$ is Boolean.
\end{proof}

\begin{body}
Proposition~\ref{prop:no-gap-1-var-easy} and Corollary~\ref{cor:dneg-has-no-gap} show that, for $n \geq 2$ and $m \geq 3$, the equations $i_n = \top$ and $d_m = \top$ have no finite satisfiability gap. Lemma~\ref{lemma:1-var-negative-direction} below thus completes our proof of Theorem~\ref{thm:1-var-classification}.
\end{body}

\begin{body}
Proposition~\ref{prop:no-gap-1-var-easy} suggests that the sequence $(H \oplus_k \top)_{k \in \mathbb{N}}$ is a good candidate for showing that the equations $i_n = \top$ and $d_m = \top$ for $n \geq 4$ and $m \geq 3$ have no gap, since the degree of satisfiability of both of these equations in $H \oplus_{k} \top$ tends to $1$ in as $k$ tends to infinity (this follows by the ordering of the Rieger-Nishimura lattice and from the fact that the degree of satisfiability of $i_2=\top$ in $H \oplus_k \top$ tends to $1$ by the proof of Proposition~\ref{prop:no-gap-1-var-easy}). However, in itself this does not suffice to show that the equation $\varphi = \top$ has no gap for all $\varphi \geq_{\textsc{RN}} i_2$, since we have not yet ruled out the case that $\ds_{H \oplus_{k} \top}(\varphi = \top) = 1$ for all $k$. In Lemma~\ref{lemma:does-not-move} we show that, starting with some $H$ such that $\ds_{H}(\varphi = \top) \neq 1$, the elements of $H$ which do not satisfy the equation $\varphi = \top$ in $H$ also do not satisfy it in $H \oplus_{k} \top$.
\end{body}

\begin{lemma}\label{lemma:does-not-move}
Let $\varphi$ be an element in the Rieger-Nishimura lattice satisfying "$\varphi \in (\uparrow_{\textsc{RN}} d_2)$ or $\varphi \in (\uparrow_{\textsc{RN}} i_2)$", $H$ be a Heyting algebra and $\iota: H \hookrightarrow H \oplus \top$ be the lattice-inclusion defined so that $(H \oplus \top) \setminus \iota(H) = \{\top\}$. For any $x \in H$, if $\varphi(x) \neq \top$, then $\varphi(\iota(x)) = \iota(\varphi(x))$.
\end{lemma}
\begin{proof}
We prove this by induction. For the base case, consider $d_2$ and $i_2$. Notice that \[\iota(i_2(\bot)) = \iota(\neg \neg \bot) = \iota(\bot) = \bot_{H \oplus \top} = \neg \neg \bot_{H \oplus \top} = \neg \neg \iota(\bot) = i_2(\iota(\bot)).\] Furthermore, for any $x \neq \bot$ such that $i_2(x) \neq \top$ (which implies that $\neg \neg x \neq \top$ and hence $\neg x \neq \bot$) we have
\begin{align*}
    \iota(i_2(x)) &= \iota( \neg \neg x)\\
    &= \neg \iota( \neg x) &\text{(by Lemma~\ref{lemma:neg-adjoin-top} and since } \neg x \neq \bot\text{)} \\
    &= \neg \neg \iota(x) &\text{(by Lemma~\ref{lemma:neg-adjoin-top} and since } x \neq \bot\text{)}\\
    &= i_2(\iota(x)).
\end{align*}
Furthermore, for all $x$ such that $d_2(x) = x \vee \neg x  \neq \top$ (which implies $x \not \in \{\bot, \top\}$) we have that
\begin{align*}
    \iota(d_2(x)) &= \iota(x \vee \neg x) \\&
    = \iota(x) \vee \iota(\neg x) &\text{(since } \iota \text{ is a lattice homomorphism)}\\
    &= \iota(x) \vee \neg \iota(x) &\text{(by Lemma~\ref{lemma:neg-adjoin-top} and since } x \neq \bot) \\
    &= d_2(\iota(x)).
\end{align*} 
This concludes the proof of the base case.

\noindent Now suppose by way of induction that the claim holds if $\varphi \in \{i_n, d_n\}$ and note throughout that if $\psi \leq_{RN} \xi$ for two formulae $\psi$ and $\xi$, then, for any $x$, if $\xi(x) \neq \top$ then $\psi(x) \neq \top$.  

\noindent If $\varphi = d_{n+1}$, then, for any $x$ such that $d_{n+1}(x) \neq \top$ we have $d_n(x) \neq \top$ and $i_n(x) \neq \top$ (since $i_n \leq_{RN} d_{n+1}$ and $d_n \leq_{RN} d_{n+1}$); thus
\begin{align*}
    \iota(\varphi(x)) &= \iota(d_{n+1}(x)) = \iota(i_n(x) \vee d_n(x)) &\text{(definition of } d_{n+1})\\
    &= \iota(i_n(x)) \vee \iota(d_n(x)) &(\text{since } \iota \text{ is a lattice homomorphism)}\\
    &= i_n(\iota(x)) \vee d_n(\iota(x)) &\text{(by induction and since } d_{n+1}(x) \neq \top)\\
    &= d_{n+1}(\iota(x)) = \varphi(\iota(x)) &\text{(definition of } d_{n+1}).
\end{align*}
If $\varphi = i_{n+1}$, then, for any $x$ such that $i_{n+1}(x) \neq \top$, we have $d_n(x) \neq \top$ (since $d_n \leq_{RN} i_{n+1}$) and also $i_n(x) \neq \top$ (since otherwise \[i_{n+1}(x) = i_n(x) \to d_n(x) = \top \to d_n(x) = d_n(x),\] which is a case we already considered). Thus we have
\begin{align}
    ~ ~ ~ ~i_{n+1}(\iota(x)) &= i_n(\iota(x)) \to d_n(\iota(x)) &\text{(definition of } i_{n+1}) \label{eqn:i-case-1}\\
    &= \iota(i_n(x)) \to \iota(d_n(x)) &\text{(induction \& since } i_{n+1}(x) \neq \top).\label{eqn:i-case-2}
\end{align}
By the definition of implication, we have that
\begin{align*}
    \iota(i_n(x)) \to \iota(d_n(x)) &:= \max \SetComp{c \in H \oplus \top}{\iota(i_n(x)) \wedge c \leq \iota(d_n(x))},
\end{align*}
and $c$ must either be an element of $\iota(H)$ or $c = \top$. It cannot be that $c =\top$ since then $i_n(x) \leq_{H} d_n(x)$ which implies that
$
    \varphi(x) = i_{n+1}(x) = i_n(x) \to d_n(x) = \top
$
(which contradicts the assumption that $\varphi(x) \neq_{H} \top$). Thus $c$ must be an element of $\iota(H)$. In this case we have (by the definition of $i_{n+1}$) that 
$
    \iota(i_n(x)) \to \iota(d_n(x)) = \iota(i_n(x) \to d_n(x)) = \iota(i_{n+1}(x)).
$
Combining this with Equations (\ref{eqn:i-case-1}) and (\ref{eqn:i-case-2}) yields $\varphi(\iota(x)) = i_{n+1}(\iota(x)) = \iota(i_{n+1}(x)) = \iota(\varphi(x))$, as desired.
\end{proof}

\begin{lemma}\label{lemma:1-var-negative-direction}
Let $n \geq 2$ and $\varphi \in \{i_n, d_{n+1}\}$. We can construct a finite Heyting algebra $H$ such that $\ds_{H \oplus_{k} \top}(\varphi = \top) < 1$ for all $k \in \mathbb{N}$, but for which  we have $\lim_{k \rightarrow \infty} \ds_{H \oplus_{k} \top}(\varphi = \top) = 1$.
\begin{proof}
Proposition~\ref{prop:no-gap-1-var-easy} shows that $i_2 = \top$ has no gap while Corollary~\ref{cor:dneg-has-no-gap} does the same for the equation $i_3 = \top$, so from here on we can assume that $\varphi \geq_{\textsc{RN}} i_2$. Since propositional intuitionistic logic is complete with respect to finite Heyting algebras, and intuitionistic logic does not prove $\varphi$, we can find a finite Heyting algebra $H$ for which $\ds_{H}(\varphi = \top) \neq 1$. By Lemma~\ref{lemma:does-not-move} we know that $\ds_{H \oplus_k \top}(\varphi = \top) < 1$ for any $k \in \mathbb{N}$. But from Proposition~\ref{prop:no-gap-1-var-easy} and the fact that $\varphi \geq_{\textsc{RN}} i_2$ we also know that $\varphi(x) = \top$ is satisfied by at least $k$ elements of $H \oplus_k \top$. Thus we have that
$$ \ds_{H \oplus_k \top}(\varphi = \top) \geq \frac{k}{   \lvert   H  \lvert    + k}$$
which tends to $1$ since the right hand side goes to $1$ as $k \to \infty$.
\end{proof}
\end{lemma}

\begin{theorem}\label{thm:1-var-classification}
An equation $\varphi(p)$ in one free variable has finite satisfiability gap precisely if it is equivalent to one of the following: $p = \top$, $\neg p = \top$ or $p \vee \neg p = \top$.
\begin{proof}
Follows immediately from Proposition~\ref{prop:1-var-trivial}, Theorem~\ref{thm:excluded-middle-gap} and Lemma~\ref{lemma:1-var-negative-direction} by considering the Rieger-Nishimura lattice (Theorem~\ref{thm:rieger-nishimura}). 
\end{proof}
\end{theorem}

\begin{figure}
\centering
\includegraphics[width=0.5\textwidth]{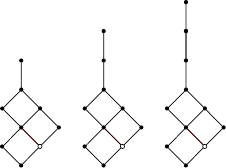} $\dots$
\caption{A sequence of Heyting algebras on which $d_5 = \top$ has no finite gap. The sole element that fails to satisfy $d_5 = \top$ is highlighted in white. Similar families can be constructed for all $\varphi \geq_\mathbf{RN} i_2$, while $i_3$ requires a different technique.}
\label{fig:heyting}
\end{figure}

\section{Classical principles in two free variables}\label{sec:2-var}

\begin{body}
The structure of free Heyting algebras in more than one generator is poorly understood: in particular, the Rieger-Nishimura theorem has no known analogue for such algebras. This prevents us from extending the methods of Section~\ref{sec:1-var} towards a complete classification of two-variable equations with respect to satisfiability gap. However, a celebrated proof-theoretic result of Pitts (Theorem~\ref{thm:pitts-quantifiers}), combined with a result of Yankov (Theorem~\ref{thm:yankov-substitution}) allows us to determine the degrees of satisfiability of many classical principles. In particular, we obtain that none of the two-variable principles listed in Example~\ref{example:classical-principles} have finite gap.
\end{body}

\begin{proposition}\label{prop:universal-reduct}
Take a theory $T$ over a first-order language $\mathcal{L}$, and a formula $\varphi(x_1,\dots,x_n,y)$ in $\mathcal{L}$. If $\varphi(x_1,\dots,x_n,y)$ has finite satisfiability gap in $T$, then so does $\forall y. \varphi(x_1,\dots,x_n,y)$.
\begin{proof}
Assume that $\varphi(x_1,\dots,x_n,y)$ has satisfiability gap $\varepsilon$. Assume that $\ds_H(\forall y. \varphi(x_1,\dots,x_n,y)) > 1 - \varepsilon$ in some finite model $H$ of $T$. We lower bound the probability that elements $w_1,\dots,w_n$ and $z$ chosen uniformly randomly from $H$ satisfy $\varphi(w_1,\dots,w_n,z)$. Choose $w_1,\dots,w_n$ and $z$ uniformly randomly from $H$. With probability exceeding $1 - \varepsilon$, the chosen $w_1,\dots,w_n$ satisfy the formula $\forall y. \varphi(w_1,\dots,w_n,y)$, and thus $w_1,\dots,w_n, z$ definitely satisfy $\varphi(w_1,\dots,w_n,z)$. But this means that $\ds_H(\varphi(x_1,\dots,x_n,y)) > 1 - \varepsilon$, and hence $\ds_H(\varphi(x_1,\dots,x_n,y)) = 1$.
\end{proof}
\end{proposition}

\begin{body}
Given an equation $\varphi(x,y)$ in two free variables, $\forall y. \varphi(x,y)$ has one free variable. One could hope to use the contrapositive of Proposition~\ref{prop:universal-reduct} to obtain non-existence of a finite satisfiability gap for a large class of equations in two variables, by reducing them to equations in one variable, which we have fully classified in Section~\ref{sec:1-var}. Unfortunately, as we see in Example~\ref{example:univ-not-equation}, in the general case, the set defined by $\forall y. \varphi(x,y)$ need not have an equational definition.
\end{body}

\begin{example}\label{example:univ-not-equation}
The set defined by the formula $\forall y. y \vee (y \rightarrow x) = \top$ need not coincide with any set defined by an equation in the language of Heyting algebras.
\begin{proof}
Consider the 4-chain as a Heyting algebra. The set $S$ defined by the formula $\forall y. y \vee (y \rightarrow x) = \top$ has two elements, and does not contain $\bot$. Since the set defined by $\neg \neg x = \top$ has 3 elements, and the set defined by $\neg x = \top$ contains $\bot$, we have that $S$ does not contain either of these sets. Using the Rieger-Nishimura theorem, we get that $S$ would have to coincide with the set defined by $x = \top$. But the latter has only one element. 
\end{proof}
\end{example}

\begin{body}
In some cases, a celebrated result of A.~M.~Pitts allows us to work around the difficulty posed by Example~\ref{example:univ-not-equation}, essentially by internalizing second-order propositional quantification in the intuitionisitic propositional calculus.
\end{body}

\begin{theorem}[Pitts \cite{pitts-quantifiers}]\label{thm:pitts-quantifiers}
Take a finite sequence of propositional variables $\overline{x}$, and a propositional variable $y$ not contained in $\overline{x}$. Let $\Phi(\overline{x}, y)$ denote a formula of intuitionistic propositional calculus containing only the variables in $\overline{x}, y$. Then we can find a propositional formula $\pita y.\Phi(\overline{x}, y)$ so that the following all hold:
\begin{enumerate}
    \item The formula $\pita y.\Phi(\overline{x}, y)$ contains only the variables in $\overline{x}$.
    \item For any propositional formula $\Psi(\overline{x})$, intuitionistic logic proves the implication $\Psi(\overline{x}) \rightarrow \pita y.\Phi(\overline{x}, y)$ precisely if it proves $\Psi(\overline{x}) \rightarrow \Phi(\overline{x}, y)$.
    \item Given any propositional formula $\Psi$, intuitionistic logic proves all implications $\pita y.\Phi(\overline{x}, y) \rightarrow \Phi(\overline{x}, \Psi)$, where $\Phi(\overline{x}, \Psi)$ denotes the formula obtained by substituting the formula $\Psi$ for the propositional variable $y$ everywhere in $\Phi$.
\end{enumerate}
\end{theorem}

\begin{body}
Keep in mind that the Pitts quantifier, $\pita$~, ~ assigns formulae of intuitionistic logic to formulae of intuitionistic logic, not first-order formulae in the language of Heyting algebras to other such formulae. Before we use them, we require an alternative to Proposition~\ref{prop:universal-reduct} which allows us to replace the universally quantified first-order formula with a Pitts quantified equation. For classical principles, Proposition~\ref{prop:pitts-reduct} constitutes one such result.
\end{body}

\begin{theorem}[Yankov \cite{yankov-classical-principles}]\label{thm:yankov-substitution}
Take a classical principle $f(x_1,\dots,x_n) = \top$ in $n$ free variables, and choose a variable symbol $p$ distinct from each of the variables $x_1,\dots,x_n$. We can find terms $y_1, \dots, y_n \in \{ \top, p, \bot \}$ so that the universal closure of the inequality $f(y_1,\dots,y_n) \leq \neg\neg p \rightarrow p$ holds in every Heyting algebra
\end{theorem}

\begin{proposition}\label{prop:pitts-reduct}
Take a classical principle $f(x,y) = \top$ in the language of Heyting algebras. If $f(x,y) = \top$ has finite satisfiability gap, then so does $\pita y. f(x,y) = \top$.
\begin{proof}
Substituting $\pita y. f(x,y)$ for $\Phi(\overline{x})$ in Theorem~\ref{thm:pitts-quantifiers} yields that intuitionistic logic proves the implication $(\pita y. f(x,y)) \rightarrow f(x,y)$. Passing through the Heyting semantics, we get that the universal closure of the inequality $\pita y. f(x,y) \leq f(x,y)$ holds in every Heyting algebra.

Now assume that the equation $f(x,y) = \top$ has finite satisfiability gap $\varepsilon$ in the class of Heyting algebras. Moreover, assume that $\ds_H(\pita y. f(x,y) = \top) > 1 - \varepsilon$ in some Heyting algebra $H$. Choose elements $a$ and $b$ uniformly randomly from $H$. With probability exceeding $1 - \varepsilon$, the chosen $a$ satisfies $\pita y. f(a,y) = \top$, and thus, by the previously derived inequality, also satisfies $f(a,b) = \top$ for all $b \in H$. But this means that $\ds_H(f(x,y) = \top) > 1 - \varepsilon$, and hence, by the finite satisfiability gap assumption, $\ds_H(f(x,y) = \top) = 1$. We conclude that $H$ is a Boolean algebra. We know from Theorem~\ref{thm:yankov-substitution} that $\pita y. f(x,y) = \top$ is equivalent to something in the downset of $i_3$, so it is either a classical principle (and thus its universal closure holds in $H$), or else equivalent to one of $\{x = \top, \neg x = \top, \bot = \top \}$, and by Proposition~\ref{prop:1-var-trivial} holds with probability no larger than $\frac{1}{2}$ in $H$.
\end{proof}
\end{proposition}

\begin{body}
Notice that while one could extend Proposition~\ref{prop:pitts-reduct} to multiple variables, the assumption that $f(x,y) = \top$ is a classical principle cannot be eliminated from the proof in any straightforward way. For example, one can find Heyting algebras where the equation $x \vee (x \rightarrow y \vee \neg y) = \top$ holds universally, but $\pita x. x \vee (x \rightarrow y \vee \neg y) = \top$ does not.
\end{body}

\begin{corollary}\label{cor:2-var-simple}
The following two-variable equations do not have finite satisfiability gap:
\begin{itemize}
    \item $((x \rightarrow y) \rightarrow x) \rightarrow x = \top$,
    \item $(\neg x \rightarrow \neg y) \rightarrow (y \rightarrow x) = \top$,
    \item $(\neg x \rightarrow y) \rightarrow (x \rightarrow y) \rightarrow y = \top$.
\end{itemize}
\begin{proof}
Observe that all of these equations have $\neg\neg x \rightarrow x = \top$ either as their universal closure or as the Pitts closure of their left-hand side (with respect to $y$), then apply Theorem~\ref{thm:1-var-classification} along with the contrapositive forms of Propositions~\ref{prop:universal-reduct}~and~\ref{prop:pitts-reduct}.
\end{proof}
\end{corollary}

\begin{body}
Corollary \ref{cor:2-var-simple}, along with Theorem~\ref{thm:1-var-classification}, settle the finite gap question for most of the commonly considered classical principles enumerated in Example~\ref{example:classical-principles}, with the sole exception of material implication, $(x \rightarrow y) \rightarrow \neg x \vee y = \top$. We shall prove that, although taking its universal closures or applying Pitts quantifiers always yields the equation $x \vee \neg x = \top$, the principle of material implication nevertheless does not have finite gap (Theorem~\ref{thm:no-gap-for-material-implication}). Consequently, the implications in Propositions \ref{prop:universal-reduct} and \ref{prop:pitts-reduct} cannot be reversed, not even in the context of classical principles.
\end{body}

\begin{definition}
The \emph{materializer} of an element $x$ in a Heyting algebra~$H$ is the set $M_H(x) = \SetComp{y \in H}{x \to y = \neg x \lor y}$.
\end{definition}

\begin{body}
We remark that materializers display interesting algebraic structure which perhaps warrants further investigation in the future. Indeed, it can be shown that they contain $\bot, \top$, are closed under $\wedge$, and absorb all right implications. Moreover, whenever we have $M_H(a) = M_H(b)$, we also have $(a \rightarrow b) \in \exclmid_H$. Accordingly, in an algebra with trivial center, all materializers (apart from those of $\bot, \top$) are unique. While we found these properties useful in other developments, none of them play any direct role in what follows; thus, to avoid unnecessary tangents, we shall not include their proofs here.
\end{body}

\begin{theorem}\label{thm:no-gap-for-material-implication}
There exists a sequence $(H_n)_{n \in \mathbb{N}}$ of Heyting algebras such that 
$$1 > \ds_{H_n}(a \to b=\neg a \lor b) \geq 1 - \frac{3^{n+1}}{(2^n + 1)^2}.$$
\begin{proof}
Take the Boolean algebra $B := (2^{[n]}, \cap, \cup, \to, \emptyset, [n])$ and let $(H, \land, \lor, \Rightarrow, \emptyset, \top)$ denote the Heyting algebra formed by adjoining a new top element $\top$ to $B$. Take any $x$ and $y$ in $2^{[n]}$ and notice that, since $x \leq y$ if and only if $x \Rightarrow y = \top$, then it must be that, if 
$x \Rightarrow y \neq \top$, then
\[x \Rightarrow y = \bigvee \SetComp{r \in H}{x \land r \leq y} = \bigvee \SetComp{r \in B}{x \land r \leq y} = x \rightarrow y.\]
This implies that, for any $x \in H$, we can lower-bound the cardinality of the materializer $M_H(x)$ of $x$ as $   \lvert   M_H(x)   \lvert     \geq    \lvert   H   \lvert    -    \lvert   \uparrow x   \lvert   $. Consequently we have
\begin{align*}
    \sum_{x \in H}   \lvert   M_H(x)   \lvert    
    &\geq    \lvert   H   \lvert   ^2 - \sum_{x \in H}    \lvert   \uparrow_H x   \lvert    &\text{(since }    \lvert   M_H(x)   \lvert    \geq    \lvert   H   \lvert    -    \lvert   \uparrow_H x   \lvert   )\\
    &=    \lvert   H   \lvert   ^2 - \Bigl (1 + \sum_{x \in B}    \lvert   \uparrow_H x   \lvert    \Bigr) &\text{(since }    \lvert   \uparrow_H \top   \lvert    = 1)\\
    &=    \lvert   H   \lvert   ^2 - \Bigl (1 + \sum_{x \in B} (1 +    \lvert   \uparrow_B x   \lvert   ) \Bigr) &\text{(since } \forall x \in H, \{\top\} = \uparrow_H x \setminus \uparrow_B x)\\
    &=    \lvert   H   \lvert   ^2 - \Bigl (1 +    \lvert   B   \lvert    + \sum_{i=0}^n \binom{n}{i}2^{n - i} \Bigr) &\text{(since } \uparrow_B x = \SetComp{s \subseteq [n]}{x \subseteq s})\\
    &=    \lvert   H   \lvert   ^2 - (1 + 2^n + 3^n) &\text{(binomial formula)}\\
    &\geq    \lvert   H   \lvert   ^2 - 3^{n+1},
\end{align*}
which then implies that $\lim_{n \to \infty} \ds_{H}(a \to b=\neg a \lor b) = 1$ since
\begin{align*}
    \ds_H(a \to b=\neg a \lor b) &= \frac{\sum_{x \in H}   \lvert   M_H(x)   \lvert   }{   \lvert   H   \lvert   ^2} \geq \frac{\lvert   H   \lvert   ^2 - 3^{n+1}}{\lvert   H   \lvert   ^2} \geq 1 - \frac{3^{n+1}}{(2^n + 1)^2}.
\end{align*}
\end{proof}
\end{theorem}

\begin{corollary}\label{cor:2-var-no-gaps}
None of the equations in two variables given in Example \ref{example:classical-principles} have a finite satisfiability gap.
\end{corollary}

\section{The infinite case}\label{sec:infinite}

\begin{body}
In group theory, the notion of \textit{index} provides an elegant way of measuring the ``density'' of subsets of infinite groups, and allows one to make sense of degree of satisfiability even in infinite groups. In fact, thanks to a relationship between virtual properties and positive degree of satisfiability~\cite{antolin-infinite-dc}, the concept manages to provide \textit{more} information about the algebraic structure of groups in the infinite setting!
\end{body}

\begin{body}
Since general Heyting algebras, unlike groups, do not come equipped with standard gadgets for measuring subset density quantitatively, we focus on finding purely \textit{qualitative} analogues of the finite satisfiability gap results, in terms of set-theoretic cardinality instead of density. We extend the one-variable classification to infinite Heyting algebras by showing that if an equation $\varphi(x)$ in one free variable has finite gap, then the set $\SetComp{x \in H}{\neg \varphi(x)}$ cannot be non-empty and finite in an infinite Heyting algebra (Corollary~\ref{cor:excluded-middle-gap-infinite}), and has to be at least as large as $\SetComp{x \in H}{ \varphi(x)}$ (Theorem~\ref{thm:lem-infinite-infinite}).
\end{body}

\begin{body}
Throughout this section let $H$ denote an arbitrary (not necessarily finite) Heyting algebra, and let $\mathbf{S} \subseteq H$ denote its set of non-central elements.
\end{body}

\begin{figure}
\centering
\begin{tikzpicture}[>=latex]
\newcommand{\blob}[7]{
    \draw (#1 + -1.7,2.8) node[below right] {Fig.~\ref{fig:blobs}.#4}; 
    
    \draw[#3] (#1 + -1.7,2)--(#1 + 2,2)--(#1 + 2,-1.6)--(#1 + -1.7,-1.6)--cycle;
    
    \draw[#2, rounded corners=3mm] 
    (#1 + -1.5,0)--
    (#1 + -1.3,.5)--
    (#1 + -.8,1.2)--
    (#1 ,1)--
    (#1 + 1.5,1.2)--
    (#1 + 1.75,.5)--
    (#1 + 0.875,-.7)--
    (#1,-1)--
    (#1 + -1.1,-.7)--
    cycle;
    
    \draw (#1,0.8) node[below right] {#5};
    \draw (#1-1,1.7) node[below right] {#6};
    \draw (#1,1.5) node[below right] {#7};
}

\blob{-3.3}{very thick, dotted, fill = parchment!60!white}{very thick}{A}{$\mathbf{s}$}{}{}
\blob{0.9}{very thick, white}{very thick, fill = parchment!60!white}{B}{}{$\sigma$}{}
\blob{5.1}{very thick, dotted}{very thick}{C}{}{}{$\mathbf{\neg \sigma}$}
\end{tikzpicture}

\caption{A (not closed) open set $s$ of a topological space $V$ is either dense ($\neg s = \bot$), or the inclusion $s \subseteq s \cup \mathrm{Int}(V \setminus s)$ is proper. In any case, $\sigma = s \cup \mathrm{Int}(V \setminus s)$ itself is dense (its complement has empty interior; $\neg \sigma = \bot$). If $s$ is a maximal non-closed open set, then $s = \sigma$.}
\label{fig:blobs}
\end{figure}

\begin{lemma}\label{lemma:max-noncentral-dense}
Any maximal element $s \in \mathbf{S}$ (i.e. an element so that for any $s' \in \mathbf{S}$, if $s \leq s'$ then $s = s'$) is dense in $H$ (i.e. $\neg s = \bot$).
\begin{proof}
Set $\sigma = s \vee \neg s$. Since $s$ is a non-central element, $\sigma \neq \top$, so we have that
$$\sigma \vee \neg \sigma = (s \vee \neg s) \vee \neg (s \vee \neg s) = s \vee \neg s \vee \bot = s \vee \neg s \neq \top,$$
and hence $\sigma \in \mathbf{S}$ too. But $s \leq \sigma$, so in fact $s = \sigma$. Hence, $\neg s = \neg \sigma = \neg (s \vee \neg s) = \bot$ as claimed. (See also Figure~\ref{fig:blobs} for a topological interpretation of this argument.)
\end{proof}
\end{lemma}

\begin{lemma}\label{lemma:vee-wedge-center}
If $a \vee b$ and $a \wedge b$ both belong to $\exclmid_H$, then so do $a$ and $b$.
\begin{proof}
We prove that $(a \vee b) \vee \neg (a \vee b), (a \wedge b) \vee \neg (a \wedge b) \vdash a \vee \neg a$ is a theorem of intuitionistic logic. By completeness it follows that if $a \vee b$ and $a \wedge b$ both belong to $\exclmid_H$, then so do $a$ and $b$ in any Heyting algebra $H$. The intuitionistic logic argument goes as follows.
From $(a \vee b) \vee \neg (a \vee b)$, we have that one of $a$, $b$ or $\neg (a \vee b)$ holds. If it is $a$, then $a \vee \neg a$ holds, so we're done. If it is $\neg (a \vee b)$, then $\neg a$ also holds, and so does $a \vee \neg a$, so we're done again. So for the rest of the argument we can assume that $b$ holds. There are two further possibilities: either $a \wedge b$ holds, or $\neg (a \wedge b)$ holds. In the former case, $a$ and hence $a \vee \neg a$ immediately follow. In the latter case, we have that $b$ holds, so if $a$ held, $a \wedge b$ would follow, contradicting $\neg (a \wedge b)$. Thus, we get that $\neg a$, and consequently $a \vee \neg a$, must hold.
\end{proof}
\end{lemma}

\begin{lemma}\label{lemma:belt}
For any $x \in H$ and any maximal non-central element $\sigma \in \mathbf{S}$, either $x \leq \sigma$ or $x \lor \sigma = \top$.
\begin{proof}
Since $\sigma$ is maximal non-central, then either $x \vee \sigma = \sigma$ or else $x \vee \sigma$ is central. In the first case $x \leq \sigma$ as desired, and in the second case, we can calculate as follows:
\begin{align*}
    \top &= (\sigma \lor x) \lor \neg (\sigma \lor x) \\
    &= (\sigma \lor x) \lor (\neg \sigma \land \neg x) \\
    &= (\sigma \lor x) \lor (\bot \land \neg x) & \text{(by Lemma~\ref{lemma:max-noncentral-dense})} \\
    &= \sigma \lor x.
\end{align*}
\end{proof}
\end{lemma}

\begin{corollary}\label{cor:belt}
For any $x \in \exclmid_H$ and any maximal non-central element $\sigma \in \mathbf{S}$, we have precisely one of $x \vee \sigma = \top$ or $\neg x \vee \sigma = \top$.
\begin{proof}
It's clear that both equalities cannot hold. By Lemma~\ref{lemma:belt}, either $x \vee \sigma = \top$ or $x \leq \sigma$. In the former case we have our claim. In the latter case, since $\neg x \vee x = \top$ and $x \leq \sigma$, we get that $\neg x \vee \sigma = \top$ as desired.
\end{proof}
\end{corollary}

\begin{theorem}\label{thm:lem-finite-infinite}
Take any maximal non-central element $\sigma \in \mathbf{S}$. The function $f: \exclmid_H \rightarrow \mathbf{S}$ given by the equation
$$f(x) = \begin{cases}
  \sigma \wedge x & \text{if $\sigma \vee x = \top$} \\
  \sigma \wedge \neg x & \text{otherwise}
\end{cases}
$$
is well-defined and two-to-one.
\begin{proof}
For well-definedness, apply Lemma~\ref{lemma:vee-wedge-center} and Corollary~\ref{cor:belt}. Now take two central elements $c,d$ such that $\sigma \vee c = \top$ and $\sigma \vee d = \top$. Assume that $f(c) = \sigma \wedge c = \sigma \wedge d = f(d)$. Then we have that 
\[c = c \vee (\sigma \wedge c) = c \vee (\sigma \wedge d) = (c \vee \sigma) \wedge (c \vee d) = \top \wedge (c \vee d) = c \vee d.\] By a similar argument, we have $d = c \vee d$. Thus $c = c \vee d = d$ and hence (applying Corollary~\ref{cor:belt} once more) the function $f$ is two-to-one as claimed.
\end{proof}
\end{theorem}

\begin{corollary}\label{cor:excluded-middle-gap-infinite}
Every non-Boolean Heyting algebra that has only finitely many non-central elements is finite.
\begin{proof}
A non-Boolean Heyting algebra that has finitely many non-central elements always has a maximal non-central element. The result follows immediately from Theorem~\ref{thm:lem-finite-infinite}.
\end{proof}
\end{corollary}

\begin{body}
Notice that Theorem~\ref{thm:lem-finite-infinite} immediately yields an alternative proof of Theorem~\ref{thm:excluded-middle-gap} (since the map $f$ in Theorem~\ref{thm:lem-finite-infinite} is two-to-one). Corollary~\ref{cor:excluded-middle-gap-infinite} is an infinitary analogue of the latter. We obtain an even stronger infinitary result in Theorem~\ref{thm:lem-infinite-infinite}.
\end{body}

\begin{theorem}\label{thm:lem-infinite-infinite}
Consider an infinite Heyting algebra $H$, and define the following subsets:
\begin{align*}
    C = \SetComp{x \in H}{x \vee \neg x = \top} \\
    D = \SetComp{x \in H}{x \vee \neg x \neq \top}
\end{align*}
We always have either $   \lvert   D   \lvert    = 0$ or $   \lvert   C   \lvert    \leq    \lvert   D   \lvert   $.
\begin{proof}
If $   \lvert   C   \lvert   <\infty$ or $   \lvert   D   \lvert   <\infty$, then the result follows from Corollary~\ref{cor:excluded-middle-gap-infinite}. For the rest of the proof, we assume that the sets $C$ and $D$ are both infinite.
Choose an element $x \in D$ and set $s = x \vee \neg x$. Then we have $\neg s = \neg (x \vee \neg x) = \bot$, so $s \vee \neg s = s \vee \bot = s \neq \top$, and therefore $s \in D$. Define the sets
\begin{align*}
    C_s = \SetComp{x \in C}{x \leq s} \\
    C^s = \SetComp{x \in C}{x \not\leq s}
\end{align*}
and notice that $   \lvert   C_s   \lvert    \leq    \lvert   C^s   \lvert   $ since $x \mapsto \neg x$ gives an injective function from $C_s$ to $C^s$. Since $C_s \cup C^s = C$ it follows by cardinal arithmetic that $   \lvert   C^s   \lvert    =    \lvert   C   \lvert   $. We now construct an injective map from $C^s$ to $D$, thus showing that $   \lvert   C   \lvert    =    \lvert   C^s   \lvert    \leq    \lvert   D   \lvert   $.
Consider the function
\begin{align*}
    f: C^s \rightarrow D~ ~ ~ ~\\
    f(x) = s \wedge x\text{.}
\end{align*}
First we show that $f$ is well-defined. Take an arbitrary $x\in C^s$ and assume for a contradiction that $f(x) \not\in D$. Then $f(x) \in C$, so that $(s \wedge x) \vee \neg (s \wedge x) = \top$. Noticing that $\neg(s \wedge x) = \neg(x \wedge s) = x \rightarrow \neg s = x \rightarrow \bot = \neg x$, we calculate as follows:
\begin{align*}
    s \vee \neg x &= s \vee \neg (s \wedge x) \\
    &= s \vee (s \wedge x) \vee \neg (s \wedge x) &\text{(absorption)} \\
    &= s \vee \top &\text{(by $s \wedge x \in C$)} \\
    &= \top.
\end{align*}
But then $x \wedge s = (x \wedge s) \vee (x \wedge \neg x) = x \wedge (s \vee \neg x) = x$, so $x \leq s$, contradicting $x \in C^s$. This shows the well-definedness of the function $f: C^s \rightarrow D$. Now we need to show that $f$ is injective. So assume that $s \wedge x = s \wedge y$ for some $x,y \in C^s$. Then we have $\neg x = \neg (s \wedge x) = \neg (s \wedge y) = \neg y$ by repeating an argument above, and since $x, y \in C$ we can negate both sides to get $x = y$.
\end{proof}
\end{theorem}

\begin{corollary}\label{cor:topology}
An infinite $T_0$ topological space is either discrete or has more non-closed open sets than clopen sets.
\begin{proof}
The lattice of open sets of a topological space forms a (complete) Heyting algebra if we define $x \rightarrow y$ as the interior of $\overline{x} \cup y$, so Theorem~\ref{thm:lem-infinite-infinite} applies.
\end{proof}
\end{corollary}

\begin{body}
We remark that for each equation $\varphi(x)$ in one free variable which does not have a finite satisfiability gap, one can use standard model-theoretic techniques to construct an infinite Heyting algebra $H_\varphi$ such that $H_\varphi \not\models \varphi$, but where only finitely many elements of $H$ fail to satisfy $\varphi$. Whether this is a coincidence, or a deep feature of degree of satisfiability in Heyting algebras, remains to be seen.
\end{body}

\section{Towards black-box lattice theory}\label{sec:black-box}

\begin{body}
Babai and Szemerédi introduced black-box groups in their seminal 1984 article \cite{babai-szemeredi-bbox}, as a uniform idealized setting for the study of randomized algorithm complexity on permutation and matrix groups. Since then, the field of black-box algebra has found applications in cryptography, and has so far come to encompass the study of black-box groups and fields, and to some extent rings and even black-box projective planes \cite{borovik-yalcinkaya-projective}.
\end{body}

\begin{body}
Borovik and Yalçınkaya pioneered a successful integrated approach \cite{borovik-yalcinkaya-bbox-attacks} to black-box algebra, characterized by the application of black-box constructions associating higher structures to black-box algebraic structures (e.g. studying a black-box ring $R$ by constructing black-box $R$-modules, or black-box fields $F$ by constructing projective planes with underlying field $F$ \cite{borovik-yalcinkaya-projective}). Many algebraic (not to mention universal algebraic) constructions associate lattices to group-like and ring-like structures: e.g. Coquand \cite{coquand-spaces} suggests that ring spectra are best treated in terms of distributive lattices in settings where obtaining prime ideals is computationally difficult\footnote{Consider simply the ring $\mathbb{Z}/n\mathbb{Z}$, where exhibiting prime ideals amounts to factorizing $n$.}. Such constructions serve to motivate the black-box algebraic investigation of lattice-like structures.
\end{body}

\begin{body}
In this section we make the basic observation that finite satisfiability gap results yield efficient one-sided Monte Carlo algorithms for black-box structures. In particular, due to Theorem~\ref{thm:excluded-middle-gap}, there is an efficient such algorithm for testing whether a black-box Heyting algebra is Boolean. Our Definition~\ref{def:bbox-ha} coincides with the generic axiomatic description of black-box algebraic structures given by Borovik and Yalçınkaya \cite{borovik-yalcinkaya-bbox-attacks}, specialized to the case of Heyting algebras.
\end{body}

\begin{definition}\label{def:bbox-ha}
A black-box Heyting algebra consists of the following data:
\begin{itemize}
    \item A finite \textit{underlying Heyting algebra} $\langle H, \wedge, \vee, \rightarrow, \bot, \top \rangle$,
    \item a \textit{set of cryptoelements} $X$, bitstrings of a fixed finite length $\ell(X)$,
    \item a \textit{decryption function} $\pi: X \rightarrow H$,
\end{itemize}
along with the following fixed algorithms operating on bitstrings:
\begin{itemize}
    \item \textbf{BB1:} a randomized algorithm that takes no inputs and produces a random cryptoelement as output, in such a way that repeatedly applying the decryption function $\pi$ to the output of the algorithm produces uniformly distributed elements over $H$ (but not necessarily over $X$!);
    \item \textbf{BB2:} algorithms $\hat{\vee}$ ($\hat{\rightarrow}$, etc.) that take as input two cryptoelements $x, y \in X$, and produce as output a cryptoelement $x \hat{\vee} y$ ($x \hat{\rightarrow} y$ etc.) such that $\pi(x \hat{\vee} y) = \pi(x) \vee \pi(y)$ (etc.);
    \item \textbf{BB2':} an algorithm that produces a cryptoelement $\hat{\bot} \in X$ such that $\pi(\hat{\bot}) = \bot$ (the similar algorithm for $\top$ could be given but is redundant);
    \item \textbf{BB3:} an algorithm that takes as input two cryptoelements $x, y$, and produces the output bitstring $1$ only if $\pi(x) = \pi(y)$, and outputs $0$ otherwise.
\end{itemize}
When considering uniformly given families of black-box structures, the additional assumption is made that these algorithms operate in time polynomial in the size of $\ell(X)$.
\end{definition}

\begin{definition}
Consider a first-order sentence $\varphi$ in the language of Heyting algebras. A \textit{one-sided black-box Monte Carlo algorithm for testing property $\varphi$} is a randomized algorithm $A$ that takes as input $\langle X, \hat{\vee}, \dots \rangle$ for a black-box Heyting algebra $\langle H, X, \pi, \hat{\vee}, \dots \rangle$, such that
\begin{itemize}
    \item if $H \models \varphi$ then $A$ returns the bitstring $1$;
    \item if $H \not\models \varphi$ then with probability larger than $\frac{1}{2}$ $A$ returns the bitstring $0$, and the bitstring $1$ otherwise;
\end{itemize}
all in time polynomial in $\ell(X)$.
\end{definition}

\begin{theorem}\label{thm:bbox-boolean-recognition}
There is a one-sided black-box Monte Carlo algorithm for testing whether a black-box Heyting algebra is a Boolean algebra.
\begin{proof}
Use \textbf{BB1} to pick a random element $x \in X$, then \textbf{BB2'} to obtain $\hat{\bot}, \hat{\top}$ and \textbf{BB2} to compute $x \hat{\vee} (x \hat{\rightarrow} \hat{\bot})$. Using \textbf{BB3}, decide if  $$\pi( x \hat{\vee} (x \hat{\rightarrow} \hat{\bot}) ) = \pi(x) \vee \neg \pi(x) = \pi(\hat{\top}) = \top.$$
Repeat the same process with another random element $y \in X$. If both $\pi(x) \vee \neg \pi(x) = \top$ and $\pi(y) \vee \neg \pi(y) = \top$ hold, output the bitstring $1$. Otherwise, output the bitstring $0$.

If $H$ is Boolean, then of course $\pi(x) \vee \neg \pi(x) = \pi(y) \vee \neg \pi(y) = \top$ always holds, so the algorithm outputs $1$ as desired. Otherwise, by Theorem~\ref{thm:excluded-middle-gap}, the probability that both $\pi(x)$ and $\pi(y)$ belong to the center of $H$ is no more than $\frac{4}{9}$, so with probability larger than $\frac{1}{2}$ the algorithm returns the bitstring $0$.
\end{proof}
\end{theorem}

\begin{body}
The idea behind Theorem~\ref{thm:bbox-boolean-recognition} clearly generalizes to any first-order property of structures that is definable by an equation with finite satisfiability gap, establishing a fruitful application of degree of satisfiability to black-box algebra.
\end{body}

\begin{problem}
Is there a one-sided black-box Monte Carlo algorithm that tests whether a black-box Heyting algebra is a Boolean algebra in time polynomial in $\ell(X)$, but does not invoke the oracle $\hat{\vee}$? What about an algorithm that does not use $\hat{\bot}$? Relatedly, can we find an equation $\varphi$ in the language $\langle \wedge, \rightarrow, \bot, \top \rangle$ (resp. the language $\langle \wedge, \vee, \rightarrow, \top \rangle$) so that $\varphi$ is a classical principle with finite satisfiability gap?
\end{problem}

\section{Discussion}\label{sec:discussion}
\subsection{Anti-exceptionalism}

\begin{body}
Anti-exceptionalism -- the belief that logic should be accepted, rejected and revised according to the same standards as other (e.g. scientific) theories -- plays a distinguished role in the leading contemporary realist metaphysical accounts of reasoning (see e.g. Part 3 of G.~Russell's \textit{Justification}~\cite{russellJustify}). In his methodological treatise, T.~Williamson~\cite{williamsonModal} asserts that purported logical laws should be judged, at least in part, based on \textit{``the fit between their consequences and what is independently known''}. As such, in Williamson's account, it is possible to have disconfirming physical evidence against purported logical laws. Namely, if a well-confirmed theory yields disconfirmed consequences according to the consequence relation of a logic, that counts as evidence against the logic. For the reader interested in anti-exceptionalism, we recommend O.~T.~Hjortland's summary~\cite{hjortlandEvidence} of Williamson's account, which gives several worked-out examples.
\end{body}

\begin{body}\label{body:williamson-objections}
Although Williamson argued that his variant of anti-exceptionalism provides evidential criteria that favors classical logic over other, competing logical theories in realist metaphysical accounts of reasoning, there are several objections to this view. Beyond the ones summarized by Hjortland~\cite{hjortlandEvidence}, we further point out the non-existence of a neutral interpretation relating the structure of logical formulae to ``disconfirming physical evidence'': what one counts as disconfirming physical evidence for a compound sentence (think e.g. $((P \rightarrow Q) \rightarrow P) \rightarrow P$) may depend on the flavor of logic one chooses to use.
\end{body}

\begin{body}
Even if the objections discussed in Paragraph~\ref{body:williamson-objections} were resolved, it is still questionable whether there could be any physical situation that would allow one to distinguish between classical and intuitionistic logic using Williamson's methodology. While investigating the full implications of our technical results to Williamson's anti-exceptionalist argument will require a longer philosophical treatise, we note that, were one able to devise a large number of such situations, one could apply Heyting algebras as a semantic model, and use our results (in particular Theorem~\ref{thm:excluded-middle-gap} and its black-box formulation, Theorem~\ref{thm:bbox-boolean-recognition}) to argue that whenever the classical laws of propositional logic fail in a given setting, the evidence disconfirming them is abundant. Conversely, Theorem~\ref{thm:strong-relativism} below shows that in all contexts where the law of excluded middle holds to \textit{any} appreciable degree, some non-tautologies of intuitionistic logic hold universally.
\end{body}

\begin{theorem}\label{thm:strong-relativism}
For any natural $n \geq 2$, let $f_n$ be one of the formulae $\{i_{n+1}, d_n\}$ in the Rieger-Nishimura lattice. There is a strictly monotone function $g$ such that, given any finite Heyting algebra $H$, if $\ds_H(d_2) > \frac{2}{g(n)}$, then $\ds_H(f_n) = 1$.
\begin{proof}
Let $g(n)$ denote the smallest integer $m$ such that some Heyting algebra $M$ with $|M| = m$ fails to satisfy $f_n$. Assume that $\ds_H(d_2) > \frac{2}{g(n)}$.  Using the inductive argument in the proof of Theorem~\ref{thm:excluded-middle-gap}, we can write $H$ as the product $H \cong H_1 \times H_2 \times \cdots \times H_k$ where all the factors $H_i$ have trivial center, and by Proposition~\ref{prop:ds-of-product-is-product-of-ds}, $\ds_{H_i}(d_2) > \frac{2}{g(n)}$ holds for each factor $H_i$. This means that $|H_i| < g(n)$, so $\ds_{H_i}(f_n) = 1$ holds for each $i$ by the definition of $g(n)$. It follows that $\ds_{H}(f_n) = 1$.
\end{proof}
\end{theorem}

\subsection{Open problems}

\begin{body}
On one hand, the classification of Theorem~\ref{thm:1-var-classification} raises many follow-up questions about the case of multiple variables. For example, are there infinitely many different equations in $n \geq 2$ free variables with finite satisfiability gap? Similarly, can the classification be extended from equations to equations in two variables where one of the variables is universally quantified? In group theory, such formulae frequently have well-behaved gaps as a consequence of Lagrange's theorem. That said, we expect that the complicated structure of the free Heyting algebras on multiple generators to make all investigation in this vein quite difficult in the Heyting algebra setting.
\end{body}

\begin{body}
On the other hand, while Theorem~\ref{thm:1-var-classification} provides a complete classification, it does not say anything about relative degrees of satisfiability: for example, it could be useful to classify the \textit{pairs} of one-variable equations, $(\varphi, \psi)$ which cannot \textit{both} hold with arbitrarily high (but $\neq 1$) probability \textit{in the same algebra}. For example, the equations $\neg x \vee \neg\neg x = \top$ and $\neg\neg x = x$ constitute one such pair. We expect that a thorough understanding of all such pairs would lead to much more elegant proofs of Theorem~\ref{thm:1-var-classification}. Similarly, one could attempt classification in more restricted (say Gödel algebras) and less restricted (modular lattices) settings: the former would have applications in intermediate logics, while the latter would be useful for the development of black box lattice theory discussed in Section \ref{sec:black-box}.
\end{body}

\begin{body}
As a long-term goal, one would wish to see a classification of those intermediate logics which can be defined by an equation with finite satisfiability gap. As a stepping stone, one would want to obtain well-behaved families of formulae with finite gap: note that e.g. the Yankov formulae do not constitute such a family.
\end{body}

\paragraph{Acknowledgements.}

Both authors contributed equally to this work. B. M. Bumpus was supported both by an EPSRC studentship, and by the European Research Council (ERC) under the European Union's Horizon 2020 research and innovation programme (grant agreement No 803421, ReduceSearch). Z. A. Kocsis acknowledges the use of CSIRO's national facilities in undertaking this research.  We are grateful to the anonymous referees for their thorough review and valuable suggestions, including an idea that streamlined the proof of Theorem~\ref{thm:excluded-middle-gap}.

\bibliography{biblio}
\bibliographystyle{plainurl}

\end{document}